\documentclass[10pt,a4paper]{article}
\usepackage{articleSetting}
\usepackage[square,authoryear]{natbib}

\begin{document}
\title{Change of numeraire in the two-marginals \\ martingale transport problem\thanks{This work is partially supported by the ANR project ISOTACE (ANR-12-MONU-0013). We would like to thank Robert Dalang for the support. We are also grateful to Stefano De Marco, Pierre Henry-Labord\`ere, David Hobson and Antoine Jacquier for their helpful remarks}}

\author{Luciano Campi\footnote{London School of Economics, Department of Statistics, United Kingdom.}\hspace{0.5cm} Ismail Laachir\footnote{ENSTA ParisTech and Zeliade Systems, France.} \hspace{0.5cm}
Claude Martini\footnote{Zeliade Systems, France.}}

\date{February 29, 2016}
\maketitle

\begin{abstract}
In this paper we apply change of numeraire techniques to the optimal transport approach for computing model-free prices of derivatives in a two periods model. In particular, we consider the optimal transport plan constructed in \cite{HobsonKlimmek2013} as well as the one introduced in  \cite{BeiglJuil} and further studied in \cite{BrenierMartingale}. We show that, in the case of positive martingales, a suitable change of numeraire applied to \cite{HobsonKlimmek2013} exchanges forward start straddles of type I and type II, so that the optimal transport plan in the subhedging problems is the same for both types of options. Moreover, for \cite{BrenierMartingale}'s construction, the right monotone transference plan can be viewed as a mirror coupling of its left counterpart under the change of numeraire. An application to stochastic volatility models is also provided.
\emph{MSC Classification 2010:} 91G20, 91G80.
\end{abstract}

\section{Introduction}

Let $\mu$ and $\nu$ be two probability measures on the positive half-line $\R^*_+ := (0,\infty)$, both with unit mean and satisfying $\mu \preccurlyeq \nu$ in the sense of the convex order, i.e. $\int f d\mu \leq \int f d\nu$ for all convex functions $f:\R^*_+ \to \mathbb R$. A classical theorem by \cite{Strassen1965} 
shows the existence of a discrete time martingale $M=(M_t)_{t=0}^2  = (1, X, Y)$ with $X \sim \mu$ and $Y \sim \nu$. Let $\mathcal{M}(\mu, \nu)$ denote the set of all possible laws for such discrete martingales with pre-specified marginals $\mu$, $\nu$. If we interpret the process $M$ as a price of a given stock, any function $C(x,y)$ can be seen as a path-dependent option written on that stock. 

Motivated by the issue of model uncertainty, there has recently been a flourishing of articles on the problem of finding a model-free upper (resp. lower) bound for the price of a given option $C$, which consists in maximizing (resp. minimizing) the expectation $\mathbb E^Q [C(X,Y)]$ with respect to all measures $Q\in \mathcal M(\mu,\nu)$. Indeed, any such measure $Q$ corresponds to some model for the price of the underlying. In the model-free setting such a price is requested to be a martingale (hence free of arbitrage) and to have pre-specified marginals $\mu$ and $\nu$, which can be deduced as usual from the observation of European Call option prices via the Breeden-Litzenberger formula. Therefore, $\mathcal M(\mu,\nu)$ is the set of natural pricing measures in this context. 

The upper bound $\sup_{Q \in \mathcal M(\mu,\nu)} \mathbb E^Q [C(X,Y)]$, for instance, corresponds essentially to the cost of the least expensive semi-static strategy that super-replicates the given payoff. The lower bound has an analogue interpretation as sub-replication price. These optimization problems have been recently tackled using an approach based on optimal transport. In this respect, \cite{BeiglJuil} perform a thorough analysis of martingale transport problems and, among other results, prove that for a certain class of payoffs the optimal probabilities are of special type, called the left-monotone and right-monotone transference plans. Later on, \cite{BrenierMartingale} provide an explicit construction of such optimal transference plans for a more general class of payoffs $C$ that satisfy the so-called \textit{generalized Spence-Mirrlees} condition:
\begin{equation} \label{spenc}
C_{xyy}>0.
\end{equation} 
Finally, \cite{HobsonKlimmek2013} construct another optimal transference plan giving the model-free sub-replication price of a forward start straddle of type II, whose payoff $| X-Y|$ does not satisfy the condition (\ref{spenc}) above.

In this paper we study the effect of change of numeraire on the martingale optimal transport approach to model-free pricing. To our knowledge, change of numeraire has never been used so far in connection to optimal transport and robust pricing. We will focus on the optimal transference plans mentioned above in the case of marginals whose support is $\mathbb R^* _+$, i.e. we will consider \emph{positive} martingales with given marginals. Our main results can be briefly stated as follows: regarding \cite{HobsonKlimmek2013} optimal coupling measure, it turns out that the change of numeraire exchanges forward start straddles of type I and type II with strike $1$, where the payoff of a forward start straddle of type I is given by $|\frac{Y}{X} -1|$.
 As consequence, this yields that the optimal transport plan in the subhedging problems is the same for both types of forward start straddles. This complements, using a different method, the results in \cite{HobsonKlimmek2013} on forward start straddles of type II. On the other hand, regarding \cite{BeiglJuil} and \cite{BrenierMartingale} left and right monotone optimal transport plans, the change of numeraire can be viewed as a mirror coupling for positive martingales. More precisely, we will show that the right monotone transport plan can be obtained with no effort from its left monotone counterpart by suitably changing numeraire. The effect of such a transformation on the generalized Spence-Mirrlees condition is also studied. Other invariance properties by change of numeraire will also be proved along the way. An extended version of the present paper can be found in \cite{LaachirThesis} PhD thesis.\medskip

The paper is structured as follows. We introduce in Section \ref{change} the change of numeraire and prove its main properties. In Section \ref{SecHK} we consider forward start straddles and extend the results in \cite{HobsonKlimmek2013} to forward start straddles of type I. In Section \ref{left}, we give an application of change of numeraire to left and right monotone transference plans for positive martingales. In the last Section \ref{sym}, we study the symmetric case where $\mu$ and $\nu$ are invariant by change of numeraire. This case includes the Black-Scholes model and the stochastic volatility models with no correlation between the spot and the volatility (cf. \cite{rtouzi96}).

\paragraph{Notations:} \begin{itemize}
\item Let $X$ be any random variable defined on some measurable space $(\Omega, \mathcal F)$. We denote by $\mathcal L_{Q} (X)$ the law of $X$ under some measure $Q$. For the expectation of $X$ under $Q$ we indifferently use the notation $\mathbb E^Q [X]$ or $Q[X]$. 
\item We denote by $\mathcal{P}=\mathcal P (\R_+^*)$ the set of all probability measures $\mu$ on $\mathbb R_+^* := (0,\infty)$, equipped with the Borel $\sigma$-field $\mathcal B(\mathbb R_+ ^*)$, and set
\[ \mathcal P_1 = \mathcal P_1 (\mathbb R_+ ^*) := \left\{ \mu \in \mathcal P : \int_{\mathbb R_+ ^*} x\mu (dx) =1\right\}.\]
The subset of all measures $\mu \in \mathcal P_1$ having a positive density, say $p_\mu$, with respect to the Lebesgue measure, is denoted by $\mathcal P_1 ^d$. 
\item If $\mu, \nu \in \mathcal{P}_1$, then $F_\mu, F_\nu$ denote their respective cumulative distribution functions. We also use the notation $\delta F$ for the difference between the two, i.e.
$$\delta F = \delta F_{\mu,\nu} = F_\nu - F_\mu .$$
\item For any function $q(x)$ we use the notation $\overline{q} (x) := 1-q (x) $, and $G_\mu (x) := \int_0^x y \mu(dy)$ for the cumulated expectation of any measure $\mu$. Finally $id$ denotes the identity function. 
\end{itemize}

\section{Change of numeraire}
\label{change}

The technique of change of numeraire was first introduced by \cite{jamshidian1989exact} in the context of interest rate models and turned out to be a
very powerful tool in derivatives pricing (see \cite{geman1995changes}, \cite[Section 2.4]{jeanblanc2009mathematical} and the other references therein for further details). Here we see that such techniques can be fruitfully transposed to a model-free setting.

We consider a two-period financial market with one riskless asset, whose price is identically equal to one, and one risky asset whose discounted price evolution is modelled by the process $(M_t)_{t=0}^2 = (1,X,Y)$. The random variables $X$ and $Y$, modelling respectively the prices at time $t=1$ and $t=2$, are defined on the canonical measurable space $(\Omega, \mathcal F)$, where $\Omega = \Omega_1 \times \Omega_2$ with $\Omega_1 = \Omega_2 = \mathbb R^* _+$ and $\mathcal F = \mathcal B(\Omega)$. For any $\omega =(\omega_1,\omega_2) \in \Omega$, we set $X(\omega)=\omega_1$ and $Y(\omega)=\omega_2$. The final ingredients of our setting are the two marginals laws $\mu$ and $\nu$, which are probability measures on, respectively, $(\Omega_1, \mathcal B(\mathbb R_+))$ and $(\Omega_2, \mathcal B(\mathbb R_+))$, so that $X$ (resp. $Y$) has law $\mu$ (resp. $\nu$). Throughout the whole paper, we will work under the following standing assumption:
\begin{assumption}
The marginals $\mu$ and $\nu$ have unit mean and satisfy $\mu \preccurlyeq \nu$ in the sense of the convex order, i.e. $\int f d\mu \leq \int f d\nu$ for all convex functions $f:\R^*_+ \to \mathbb R$. \end{assumption}

Let $\mathcal{M}(\mu, \nu)$ denote the set of all probability measures on $(\Omega, \mathcal F)$ such that $X \sim \mu$, $Y\sim \nu$, and $M$ is a martingale. As we already claimed in the introduction, by a classical theorem in \cite{Strassen1965}, we know that the previous assumption guarantees that such a set is non-empty.

\subsection{The one-dimensional symmetry operator \texorpdfstring{$S$}{}}

As a preliminary step, we first consider the change of numeraire in a static setting, i.e. for the marginal laws. Thus, we define the (marginal) symmetry operator $S$ as an operator acting on the space of probability measures on $(\mathbb R_+ ^* , \mathcal B(\mathbb R_+ ^*))$ given by
\begin{equation}\label{S} S(\mu) := \mathcal L_{\bar \mu}(1/X), \quad \mu \in \mathcal P(\mathbb R_+ ^*),\end{equation}
where $\bar \mu$ is the probability measure defined by $\bar \mu (A) = \mu(X \mathbf 1_A)$, for any $A \in \mathcal B(\mathbb R^*_+)$. 

\begin{remark} \emph{Financially speaking, $S(\mu)$ is the law of the riskless asset price at time $t=1$ measured in units of the risky one under the new probability $Xd\mathbb P$. This is the usual change of measure associated to a change of numeraire. An analogue interpretation applies to $S(\nu)$.}\end{remark}

Notice that if $\mu \in \mathcal P_1$, i.e. it has unit mean, then $S(\mu) \in \mathcal P_1$ too, since $S(\mu)[X] = \mu [X/X]=1$. In the case where $\mu \in \mathcal{P}^d _1$ with density $p_\mu$, the new measure $S(\mu)$ has a density too and this is given by
\begin{equation}\label{pSmu}
p_{S(\mu)}(x) = \frac{p_\mu(1/x)}{x^3},\quad x>0,
\end{equation}
hence in particular we have $S(\mu) \in \mathcal P^d _1$. Moreover, $S$ is an involution, i.e. $S\circ S =id$. Indeed, we have 
\[ S \circ S (\mu)[f(X)] = S(\mu)[Xf(1/X)] = \mu[(X/X) f(X)] = \mu[f(X)],\]
for all bounded measurable functions $f$. For future reference we summarize our findings in the following lemma, which also contains few more properties, such as the fact that the operator $S$ preserves the convex order.
\begin{lemma}\label{symmetry}
The symmetry operator $S$ defined in (\ref{S}) satisfies the following properties: 
\begin{enumerate}
\item $S$ is an involution preserving the convex order in $\mathcal{P}_1$, i.e. $S \circ S =id$ and if $\mu, \nu \in \mathcal{P}_1$ satisfy $\mu \preccurlyeq \nu$, then $S(\mu)\preccurlyeq S(\nu)$.
\item If $\mu$ has density $p_\mu$, the measure $S(\mu)$ has a density given by $p_{S(\mu)}$ in (\ref{pSmu}).
\item If $\mu \in \mathcal{P}_1$, then for all $y>0$ we have\begin{equation*}
F_{S(\mu)}(y) = 1- G_\mu(1/y) \quad \text{and}\quad G_{S(\mu)}(y) = 1- F_\mu(1/y).
\end{equation*}
\end{enumerate}
\end{lemma}

\begin{proof}
To prove property 1 it suffices to show that $S$ preserves the convex order of measures. Let $\mu, \nu \in \mathcal{P}_1$ such that for any convex function $f$, $\int f d\mu \leq \int f d\nu$. Since $S(\mu)$ and $S(\nu)$ have both unit mass and the same first moment, it is enough to show that for any positive constants $K, L$ we have  
$$ S(\mu)[ (K X -L)_+ ] \leq S(\nu)[ (K X-L)_+ ] .$$
Now $S(\mu)[ (K X -L)_+ ]  = \mu[ X(K/X -L)_+] = \mu[(K-LX)_+]$, and the same holds true for $\nu$. Since $x \mapsto (K-Lx)_+ $ is a convex function, the result follows. Property 2 has already been proved above, so it remains to show property 3. We show only the left-hand side equality, the same arguments can be applied to get the other one. By the definition of $S$ we have
\begin{eqnarray*}
F_{S(\mu)} (y) &=& S(\mu)(X \le y) = \mu[X \mathbf 1_{(1/X \le y)}] \\
&=& \mu[X] - \mu[X \mathbf 1_{(X \le 1/y)}]1= 1-G_\mu (1/y).
\end{eqnarray*}
Hence, the proof is complete.
\end{proof}

\subsection{The symmetric two-marginals martingale problem}

In this subsection, we consider the change of numeraire in the two-period setting. Let $\S$ be the operator that assigns to every $Q \in \mathcal{M}(\mu, \nu)$ the measure $\S(Q)$ defined by
\begin{equation}\label{SS} \mathbb E^{\S(Q)}[f(X, Y)] = \mathbb E^Q\left[Y f\left(\frac{1}{X}, \frac{1}{Y}\right) \right], \; \text{for every bounded measurable function $f$}.\end{equation}

\begin{lemma}\label{PropSymmetry}
The operator $\S$ satisfies the following properties:\begin{enumerate}
\item $\S(Q)$ is a probability in $\mathcal{M}(S(\mu), S(\nu))$ and it satisfies $\S\circ\S=id$, i.e. $\S$ is an involution.
\item $\S\left(\mathcal{M}(\mu, \nu)\right) = \mathcal{M}(S(\mu), S(\nu))$.
\end{enumerate}
\end{lemma}

\begin{proof}\begin{enumerate}
\item First, let us prove that $\S(Q) \in \mathcal{M}(S(\mu), S(\nu))$ for $Q\in\mathcal{M}(\mu, \nu)$.
The fact that $Y$ has law $S(\nu)$ under $\S(Q)$ follows from the definition of $S$. Regarding $X$, by the martingale property under $Q$, we have
$$\mathbb E^{\S(Q)}[ f(X)] = \mathbb E^Q\left[Y f\left(\frac{1}{X}\right)\right] =  \mathbb E^Q\left[X f\left(\frac{1}{X}\right)\right],$$
for all bounded measurable functions $f$ depending only on $x$. Hence we conclude since $X$ has law $\mu$ under $Q$.
It remains to show the martingale property:
$$\mathbb E^{\S(Q)}[ Y f(X)] = \mathbb E^Q \left[Y \frac{1}{Y} f\left(\frac{1}{X}\right)\right] = \mathbb E^Q \left[ f\left(\frac{1}{X}\right)\right] = \mathbb E^Q \left[ X \frac{1}{X} f\left(\frac{1}{X}\right)\right].$$
Now by the martingale property under $Q$ we obtain $\mathbb E^Q [ Y \frac{1}{X} f(\frac{1}{X})]= \mathbb E^{\S(Q)}[ X f(X)] $, which implies $\mathbb E^{\S(Q)}[ Y| X] = X$. The fact that $\S$ is an involution follows immediately from its definition.
\item In order to prove that $\S\left(\mathcal{M}(\mu, \nu)\right) = \mathcal{M}(S(\mu), S(\nu))$, we note that one inclusion is implied by the property 1 in this proposition. The other inclusion is a consequence of the fact that the symmetry operator $S$ is an involution.
\end{enumerate}
\end{proof}

\begin{remark}
\emph{Notice that the symmetry operator $S$ can be seen as the projection of $\S$. Indeed, let $Q \in \mathcal M(\mu,\nu)$. For any bounded measurable function $f: \mathbb R_+ ^* \to \mathbb R$, we have $\mathbb E^{\S(Q)}[f(X)] = \mathbb E^Q\left[Y f(1/X) \right] = \mathbb E^Q [Xf(1/X)] = S(\mu)[f(X)]$, where the second equality is due to the martingale property. Hence the projection of $\S(Q)$ into the first coordinate of the product space $\mathbb R_+ ^* \times \mathbb R_+ ^*$ equals $S(\mu)$. Similarly one can see that the projection of $\S(Q)$ onto the second coordinate is $S(\nu)$.}
\end{remark}\medskip

Let $C: (\R_+^*)^2 \to \mathbb R$ be any continuous function with linear growth, i.e. $|C(x,y)| \le \kappa (1+x+y)$ for some constant $\kappa>0$. The lower and upper model-free price bounds for such a derivative can be computed by solving the following martingale optimal transport problems:
\begin{equation} \label{sup}
\underline{P}(\mu, \nu, C) := \inf_{Q \in \mathcal{M}(\mu, \nu)} \mathbb E^Q [C(X, Y)], \quad \overline{P} (\mu, \nu, C) = \sup_{Q \in \mathcal{M}(\mu, \nu)} \mathbb E^Q [C(X, Y)].
\end{equation}
They have the interpretation of sub and super-replication prices of the payoff $C$ through a duality theory that has been developed during the last few years by several authors (see Remark \ref{duality}).

The following proposition shows the symmetry properties of such model-free bounds with respect to the change of numeraire transformation.  

\begin{proposition}\label{PropSymmetry2}
Let us define the payoff $\S^*(C)(x, y) := y C(\frac{1}{x}, \frac{1}{y})$ for $x,y > 0$. Then
\begin{equation}
\underline{P}(S(\mu), S(\nu), \S^*(C)) = \underline{P}(\mu, \nu, C) , \quad \overline{P}(S(\mu), S(\nu), \S^*(C)) = \overline{P}(\mu, \nu, C) .
\end{equation}
\end{proposition}
\begin{proof}
We only prove the equality for $\overline P$, the one for $\underline P$ can be shown using the same arguments. By the definition of $\S^*(C)$ we have 
\[
\overline{P}(S(\mu), S(\nu), \S^*(C)) = \sup_{Q \in \mathcal M(S(\mu),S(\nu))} \mathbb E^Q [\S^*(C)(X,Y)] = \sup_{Q \in \mathcal M(S(\mu),S(\nu))} \mathbb E^Q [Y C(1/X,1/Y)].\]
Using property 2 in Lemma \ref{PropSymmetry} and the definition of $\S^*(Q)$, we get
\begin{eqnarray*}
 \sup_{Q \in \mathcal M(S(\mu),S(\nu))} \mathbb E^Q [Y C(1/X,1/Y)] &=& \sup_{Q \in \S(\mathcal M(\mu,\nu))} \mathbb E^{Q} [YC(1/X,1/Y)] \\
&=& \sup_{Q \in \mathcal M(\mu,\nu)} \mathbb E^{\S(Q)} [Y C(1/X,1/Y)] \\
&=& \sup_{Q \in \mathcal M(\mu,\nu)} \mathbb E^{Q} [C(X,Y)]\\
&=& \overline{P}(\mu, \nu, C),
\end{eqnarray*}
which gives the result.
\end{proof}

We conclude this section by showing how the symmetry operator $\mathbb S^*$ introduced in Proposition \ref{PropSymmetry2} acts on the space of hedgeable claims, which we define as
\begin{eqnarray*}
\mathcal H(\mu,\nu) &=& \big\{C: (\mathbb R_+^*)^2 \to  \mathbb R  :  \text{ there exist } \varphi\in \mathbb{L}^1(\mu), \;\psi\in \mathbb{L}^1(\nu),\; h\in \mathbb{L}^0 ,\\
&& \;\; C(x,y) = \varphi(x) + \psi(y) + h(x)(y-x) \;Q-\text{a.e.} \; \forall Q \in \mathcal{M}(\mu,\nu) \nonumber \big\}.
\end{eqnarray*}
This set contains all the payoffs that can be replicated by investing semi-statically in the stock as well as in Vanilla options. It turns out that this set is invariant by the symmetry operator $\S^*$ or, in other words, the set of semi-static portfolios does not depend on the choice of the numeraire. 
\begin{proposition}
The set $\mathcal H(\mu,\nu)$ is invariant by $\S^*$, i.e. $\S^*(\mathcal H(\mu,\nu)) = \mathcal H(S(\mu),S(\nu)) $.
\end{proposition}
\begin{proof}
Let $C \in \mathcal H(\mu,\nu)$, i.e. there exist functions $\varphi\in \mathbb{L}^1(\mu), \;\psi\in \mathbb{L}^1(\nu),\; h\in \mathbb{L}^0$ such that
$$C(x,y) = \varphi(x) + \psi(y) + h(x)(y-x) \quad Q-\text{a.e.} \quad \forall Q \in \mathcal{M}(\mu,\nu).$$ 
Let $\S^*(C)(x,y) := y C(1/x,1/y)$ for all $x,y >0$ and let
$$\tilde{\varphi}(x) = x\varphi(1/x), \;
\tilde{\psi}(y) = y \psi(1/y), \;
\tilde{h}(x) = \left(\varphi(1/x)-1/xh(1/x)\right), \quad x,y>0.$$
Such functions verify  $\tilde{\varphi}\in \mathbb{L}^1(S(\mu)), \;\tilde{\psi}\in \mathbb{L}^1(S(\nu)),\; \tilde{h}\in \mathbb{L}^0$.
We can check by direct computation that
\[ \S^*(C)(x,y) = \tilde{\varphi}(x) + \tilde{\psi}(y) + \tilde{h}(x) (y-x), \quad Q - a.e. \; \forall Q \in \mathcal M(S(\mu),S(\nu)). \]
Furthermore, since $\S(\mathcal{M}(\mu,\nu)) = \mathcal{M}(S(\mu),S(\nu))$, 
we have the following equivalences:
\begin{eqnarray*}
&& \mathbb{E}^Q\left[ \left| C(X,Y) -\varphi(X) - \psi(Y) - h(X)(Y-X)\right| \right] = 0, \quad \forall Q \in \mathcal{M}(\mu,\nu) \\
&\Leftrightarrow & \mathbb{E}^{\S(Q)}\left[ \left| \S^*(C)(X,Y) -\tilde\varphi(X) - \tilde\psi(Y) - \tilde{h}(X)(Y-X)\right| \right] = 0,\quad \forall Q \in \mathcal{M}(\mu,\nu) \\
&\Leftrightarrow & \mathbb{E}^Q\left[ \left| \S^*(C)(X,Y)-\tilde\varphi(X) - \tilde\psi(Y) - \tilde{h}(X)(Y-X) \right| \right] = 0, \quad \forall Q \in \mathcal{M}(S(\mu),S(\nu)).
\end{eqnarray*}
Hence $$\S^*(C)(x,y)=\tilde\varphi(x) + \tilde\psi(y) + \tilde{h}(x)(y-x), \;Q-\text{a.e.}, \; \forall Q \in \mathcal{M}(S(\mu),S(\nu)),$$
i.e. $\S^*(C) \in \mathcal H(S(\mu),S(\nu))$.
\end{proof}

\section{Model-free pricing of forward start straddles}
\label{SecHK}

In this section we apply our results on the change of numeraire to compute the model-free sub-replication price of a forward start straddle of type I, which complements the result obtained in \cite{HobsonKlimmek2013}. 

In their article \cite{HobsonKlimmek2013} consider the problem of computing a model-free lower bound on the price of an option paying $|Y-X|$ at maturity. This is an example of \emph{type II forward start straddle}, whose payoff for any strike $\alpha >0$ is given by
\begin{equation}
C^\alpha_{II}(x,y)=\left|y-\alpha x \right|, \quad x,y>0,
\end{equation}
while the \emph{type I forward start straddle} with strike $\alpha >0$ is given by 
\begin{equation}\label{FSS1}
C^\alpha_{I}(x,y)=\left|\frac{y}{x}-\alpha \right|, \quad x,y>0,
\end{equation}
cf. \cite{Lucic2004forward} and \cite{Jacquier2012asymptotics}.
\cite{HobsonKlimmek2013} derive explicit expressions for the coupling minimizing the model-free price of an at-the-money (ATM) type II forward start straddle $C^1_{II}$ as well as for the corresponding sub-hedging strategy. In particular, they show that the optimal martingale coupling for such a derivative is concentrated on a three points transition $\{p(x), x, q(x)\}$ where $p$ and $q$ are two suitable decreasing functions. The precise result will be recalled below. Such a characterization is obtained under a \textit{dispersion assumption} \cite[Assumption 2.1]{HobsonKlimmek2013} on the supports of the marginal laws: the support of $(\mu-\nu)^+$ is contained in a finite interval $E$ and the support of $(\nu-\mu)^+$ is contained in its complement $E^c$. Instead of working under such a condition on the supports, we would rather impose the following standing assumption.
\begin{assumption}
\label{assSingleMax}Let the following properties hold:\begin{enumerate}
\item[(i)] The measures $\mu$ and $\nu$ belong to $\mathcal P_1 ^d$;
\item[(ii)] $\delta F$ has a single local maximizer $m$.
\end{enumerate}
\end{assumption}

The main reason for working under this assumption in the rest of the paper is twofold: first, it makes our treatment more uniform, since later we will consider \cite{BrenierMartingale} construction of the right and left monotone transference plans and their construction holds if the marginals are absolutely continuous, whence our Assumption \ref{assSingleMax}(i). Moreover, in the case of marginals with densities, Assumption \ref{assSingleMax}(ii) is equivalent to the dispersion assumption in \cite{HobsonKlimmek2013} (as we show in Remark \ref{assumpHK} below) and it simplifies the study of \cite{BrenierMartingale} construction in the next section.

\begin{remark}\label{assumpHK}
\emph{Let $\mu, \nu \in \mathcal{P}^d _1$ with $\mu \preccurlyeq \nu$. Then Assumption 2.1 in \cite{HobsonKlimmek2013} is equivalent to our Assumption \ref{assSingleMax}(ii). To see this, let $\mu, \nu \in \mathcal{P}_1$ with $\mu \preccurlyeq \nu$. First, observe that 
\[ \textrm{supp} ((\mu-\nu)^+) = \text{cl} \left\{x \in \mathbb R_+^* : \int_{(x-\epsilon, x+\epsilon)} (p_\mu(z)-p_\nu(z)) dz >0, \text{for some $\epsilon>0$}\right\},\]
where $\text{cl}(A)$ denotes the closure of any subset $A \subset \mathbb R_+ ^*$.
Suppose that Assumption 2.1 in \cite{HobsonKlimmek2013} holds, i.e. there exist constants $0 \le a < b$ such that 
\[p_\mu(x)-p_\nu(x) > 0 \text{ for } x \in (a,b) \quad \text{and}\quad  p_\mu(x)-p_\nu(x)\leq 0 \text{ for } x\in (a,b)^c.\]
Consequently, $\delta F$ is decreasing on $[a,b]$ and increasing on $(0,a)$ and $(b, \infty)$. Hence it admits a unique maximizer at $a$ and a unique minimizer at $b$, whence Assumption \ref{assSingleMax} follows.
Conversely, suppose that Assumption \ref{assSingleMax} holds. Then $\dF$ admits a global maximum in $m>0$. Moreover, by the convex order of $\mu$ and $\nu$, $\dF$ admits a global minimum at $\tilde m > m$. Hence, for all $x\in(m,\tilde m)$ we have $p_\mu(x)-p_\nu(x) > 0$, while for all $x\in(m,\tilde m)^c$ we have $p_\mu(x)-p_\nu(x) \leq 0$, and finally Assumption 2.1 in \cite{HobsonKlimmek2013} is fulfilled.}
\end{remark}

\begin{remark}\label{ass-preserve}
\emph{Both properties in Assumptions \ref{assSingleMax} are preserved under change of numeraire. Indeed, we have already seen in Lemma \ref{symmetry} that $S(\mu),S(\nu)$ belong to $\mathcal P_1 ^d$. Concerning property (ii) in the assumption, note that
$$F_{S(\mu)}(y) = \int_0^y \frac{p_\mu(\frac{1}{x})}{x^3} dx = 1- \int_0^{1/y} x p_\mu(x) dx,$$ 
so that $$\delta F_S(y)= F_{S(\nu)}-F_{S(\mu)} = - \int_0^{1/y} x \partial_x (\delta F)(x) dx .$$
Hence, $\delta F_S$ has a single maximizer $x^S_\star$ if and only if $\delta F$ has a single minimizer $x^\star$, satisfying $x^\star=\frac{1}{x^S_\star}$. }
\end{remark}

Let us come back to the model-free pricing of forward start straddles. Given the form of the payoff (\ref{FSS1}), it is very natural to try to obtain an optimal martingale coupling for its model-free sub-hedging price combining the change of numeraire techniques with \cite{HobsonKlimmek2013} results. For reader's convenience, we summarize their main result in the following theorem. It is a consequence of Theorem 5.4 and Theorem 5.5 in \cite{HobsonKlimmek2013} applied to the particular case when the marginals $\mu,\nu$ have densities (see their Subsection 6.1). Therefore, its proof is omitted.

\begin{theorem}
\label{ExistenceUniqHK}
Let Assumption \ref{assSingleMax} hold. Then there exists a unique optimal coupling $\Q_{HK}(\mu,\nu) \in \mathcal{M}(\mu, \nu)$ such that
\begin{equation} \label{infFwdStartII}
\underline{P}(\mu, \nu, C^1_{II}) := \inf_{Q \in \mathcal{M}(\mu, \nu)} \mathbb{E}^Q\left[ |Y-X| \right]= \mathbb{E}^{\Q_{HK}(\mu,\nu)}\left[|Y-X|\right].
\end{equation}
Moreover, $\Q_{HK}(\mu,\nu)(dx,dy) = \mu(dx)\mathcal L_{HK}(x,dy)$, with a transition kernel $\mathcal L_{HK}$ given by
\begin{equation}
\label{kernelHK}
\mathcal{L}_{HK}(x,\cdot) = \delta_x \1_{x\leq a} + (l(x)\delta_{p(x)} + u(x)\delta_{q(x)} + (1-l(x)-u(x)) \delta_{x} ) \1_{a< x< b} + \delta_{x} \1_{x\geq b},\end{equation}
where: \begin{enumerate}
\item $a$ (resp. $b$) is the global maximizer (resp. minimizer) of  $\delta F$;
\item $p: (a,b)\rightarrow [0,a]$ and $q: (a,b) \rightarrow [b,\infty]$ are continuous decreasing functions solutions to the equations
\begin{equation}
\begin{aligned}
\label{equatioHK}
\delta F(q(x)) + \delta F(p(x)) &= \delta F(x), \\
\delta G(q(x)) + \delta G(p(x)) &= \delta G(x), \quad x \in (a,b).
\end{aligned}
\end{equation} 
\item $l,u: (a,b) \rightarrow [0,1]$ are given by
\begin{equation}
\begin{aligned}
u(x) &= \frac{x-p(x)}{q(x)-p(x)} \frac{p_\mu(x)-p_\nu(x)}{p_\mu(x)}, \\
l(x) &= \frac{q(x)-x}{q(x)-p(x)} \frac{p_\mu(x)-p_\nu(x)}{p_\mu(x)}.
\end{aligned}
\end{equation}
\end{enumerate}  
\end{theorem}

Now, a simple application of change of numeraire results from the previous section gives that $\Q_{HK}(\mu,\nu)$ attains the lower bound price for the type I forward start straddle $C_I^1$ as well. This result complements the one in \cite{HobsonKlimmek2013} about type II forward start straddle $C_{II}^1$. We show first a symmetry property of Hobson-Klimmek optimal coupling.

\begin{proposition}
\label{HKSymm}
The martingale measure $\Q_{HK}(\mu,\nu)$ verifies the symmetry relation
\begin{equation*}
\S \left(\Q_{HK}(S(\mu), S(\nu))\right) = \Q_{HK}(\mu,\nu)
\end{equation*}
where the symmetry operator $\S$ is defined in \ref{SS}.
\end{proposition}

\begin{proof}
Let the pair $(p^S,q^S)$ define the measure $\Q_{HK}(S(\mu), S(\nu))$.
A simple computation shows that the measure $\mathbb S( \Q_{HK}(S(\mu), S(\nu)))$ is concentrated on $\{\frac{1}{p^S(1/x)}, x, \frac{1}{q^S(1/x)}\}$. In order to get the equations satisfied by this three-band graph, recall first the symmetry relations
\begin{equation}\label{FGsymm}
\dF^{S}(y) = -\dG(1/y), \quad \dG^{S}(y) = -\dF(1/y).
\end{equation}
By definition, $(p^S,q^S)$ is characterized by the two equations
\begin{eqnarray*}
\dF^S(q^S(x)) + \dF^S(p^S(x)) &=& \delta F^S(x), \\
\dG^S(q^S(x)) + \dG^S(p^S(x)) &=& \delta G^S(x).
\end{eqnarray*}
Hence, using (\ref{FGsymm}) we have
\begin{eqnarray*}
\dF(1/q^S(1/x)) + \dF(1/p^S(1/x)) &=& \dF(x), \\
\dG(1/q^S(1/x)) + \dG(1/p^S(1/x)) &=& \dG(x).
\end{eqnarray*}
Since the functions $x \mapsto 1/p^S(1/x)$ and $x \mapsto 1/q^S(1/x)$ are both continuous decreasing and satisfy the same equations as the pair $(p,q)$, they are candidates. Hence, the uniqueness of the optimal coupling yields the result.
\end{proof}

At this point we can exploit a symmetry relation between type I and type II forward start straddles, which is given by
\begin{equation}
\label{fwdStraddleI_II}
\S^*(C^\alpha_{II})(X,Y) = Y\left|\frac{1}{Y}-\frac{\alpha}{X}\right| = \alpha\left|\frac{Y}{X}-\frac{1}{\alpha}\right| = \alpha C^{\frac{1}{\alpha}}_{I}.
\end{equation}
In particular, the ATM straddles, i.e. $\alpha =1$, are related by $\S^*(C^1_{II})(X,Y) = C^1_{I}(X,Y)$. A consequence of this is the following proposition, that states the announced result on forward start straddle of type I and concludes the section.

\begin{proposition}
The lower bound price of the ATM forward start straddle of type I is also attained by optimal coupling $\mathbb Q_{HK}(\mu,\nu)$, i.e.
\begin{equation} \label{infFwdStartI}
\underline{P}(\mu, \nu, C^1_{I}) := \inf_{Q \in \mathcal{M}(\mu, \nu)} \mathbb{E}^Q \left[\left|\frac{Y}{X}-1\right| 
\right]= \mathbb E^{\Q_{HK}(\mu, \nu)} \left[ C^1_{I} \right].
\end{equation} 
\end{proposition}
\begin{proof} Using Proposition \ref{PropSymmetry2} and the relation (\ref{fwdStraddleI_II}), we have 
\begin{eqnarray*}
\underline{P}(\mu, \nu, C^1_{I}) &=& \underline{P}(S(\mu), S(\nu), \S^*(C^1_{I}))\\
&=& \underline{P}(S(\mu), S(\nu), C^1_{II} )\\
&=& \mathbb E^{\Q_{HK}(S(\mu), S(\nu))} \left[ C^1_{II}\right] \\
&=& \mathbb E^{\Q_{HK}(\mu, \nu)} \left[ C^1_{I} \right].
\end{eqnarray*}
The proof is therefore complete.
\end{proof}

\section{Symmetry properties of left and right monotone transference plans}\label{Sec:HLT}
\label{left}

The optimization problems in \eqref{sup} are strongly related to the concepts of \textit{right} and \textit{left monotone transference plans}. Both notions were introduced in \cite{BeiglJuil}, who show their existence and uniqueness for convex ordered marginals, and prove that they solve the maximization and the minimization problem in \eqref{sup} for a specific set of payoffs of the form $C(x,y) = h(y-x)$ with $h$ differentiable with strictly convex first derivative. \cite{BrenierMartingale} extend these results to a wider set of payoffs. Moreover they also give an \textit{explicit} construction of the left-monotone transference plan. In this section we want to study the symmetric property of those transference plans and show in particular that, in the case of positive martingales, the right monotone plan can be obtained from its left monotone counterpart with no effort via change of numeraire.

We start by recalling the general definition of right and left monotone transference plan.

\begin{definition}[\cite{BeiglJuil}]
A martingale measure $Q \in \mathcal M(\mu,\nu)$ is \emph{left-monotone} (resp. \emph{right-monotone}) if there exists a Borel set $\Gamma \subset (\mathbb R_+ ^*)^2$ with $Q(\Gamma)=1$ such that for all $(x,y^-),(x,y^+)$ and $(x',y')$ in $\Gamma$ we cannot have $x<x'$ and $y^- < y' < y^+$ (resp. $x>x'$ and $y^- < y' < y^+$). We denote $Q_L(\mu,\nu)$ (resp. $Q_R (\mu,\nu)$) the left-monotone (resp. right-monotone) transference plan with marginals $\mu,\nu$.
\end{definition}

The next result states how the two monotone transference plans relate to each other via the symmetry operators.

\begin{proposition}\label{LM-RM}
The operator $\S$ exchanges left-monotone and right-monotone transference plans, i.e. $\S(Q_R (S(\mu),S(\nu))) = Q_L(\mu,\nu)$ and $\S(Q_L(S(\mu),S(\nu)))=Q_R(\mu,\nu)$.
\end{proposition}

\begin{proof}
We prove only the first equality, as the second follows immediately since $\S$ is an involution. By definition of the right-monotone transference plan $Q_R ^S := Q_R(S(\mu),S(\nu))$, there exists a Borel set $\Gamma_R \subset (\mathbb R_+ ^*)^2$ such that $Q^S _R (\Gamma_R) =1$ and for all $(x,y^-),(x,y^+),(x',y')$ in $\Gamma_R$ we cannot have $x > x'$ and $y^- < y' < y^+$. Let
\[ \Gamma^\S _R := \{ (x,y) \in (\mathbb R_+ ^*)^2 : (1/x,1/y) \in \Gamma_R \}.\]
We clearly have
\[ \S(Q_R ^S) (\Gamma^\S _R) = Q^S _R [Y\mathbf 1_{\Gamma^\S _R}(1/X,1/Y)] = Q^S _R [Y \mathbf 1_{\Gamma_R}(X,Y)] = Q^S _R(Y) = 1.\]
Moreover, since $(x,y^-),(x,y^+),(x',y') \in \Gamma^\S _R$ if and only $(1/x,1/y^-),(1/x,1/y^+),(1/x',1/y') \in \Gamma_R$, we cannot have $x < x'$ and $y^- < y' < y^+$. Therefore, we have $\S(Q_R (S(\mu),S(\nu))) \in \mathcal M(\mu,\nu)$, hence by uniqueness of the left-monotone transference plan (see Theorem 1.5 in \cite{BeiglJuil}) we obtain $\S(Q^S _R)=Q_L(\mu,\nu)$. 
\end{proof}

\begin{remark}
\emph{We observe that, as a by-product of the previous result, the existence of left-monotone transference plan gives for free the existence of its right-monotone analogue via the symmetry operator $\S$ and vice-versa. Moreover, notice also that the result above holds in full generality, e.g. even when the marginals do not have densities.} 
\end{remark}

Building on the results in \cite{BeiglJuil}, \cite{BrenierMartingale} show in particular that $\Q_L(\mu,\nu)$ attains the upper bound \eqref{sup} for a larger class of payoffs verifying a generalized Spence-Mirrlees type condition $C_{xyy}>0$ (or $C_{xyy}<0$) (see their Theorem 5.1). We summarize their result in the following theorem.\footnote{Observe that the results in Theorem \ref{ThmHLT} hold under more general conditions than our Assumption \ref{assSingleMax}(ii).} \medskip

\begin{theorem}[\cite{BrenierMartingale}] \label{ThmHLT}
Let $C : (\mathbb R_+ ^*)^2 \to \mathbb R$ be a measurable function such that the partial derivative $C_{xyy}$ exists and $C_{xyy}>0$. Under Assumption \ref{assSingleMax}, the left-monotone transference plan $Q_L=Q_L (\mu,\nu)$ is the optimal coupling solving the martingale transport problem
\[ \overline P (\mu,\nu,C) := \sup_{Q \in \mathcal M(\mu,\nu)} \mathbb E^Q[C(X,Y)].\] 
\end{theorem}

In order to apply the change of numeraire approach, notice first that by the definition of $\S^*(C)$ we have
\begin{equation}\label{sym-SM}
\S^*(C)_{xyy}(x, y) = -\dfrac{1}{x^2y^3}C_{xyy}\left(\frac{1}{x}, \frac{1}{y}\right), \; \forall x,y>0.
\end{equation}
Hence, we have that $C_{xyy}>0$ holds true if and only if $\S^*(C)_{xyy}<0$. This elementary remark allows to find the model-free price bounds for payoffs verifying $C_{xyy}<0$ by changing the numeraire. This is similar to what happens with the mirror coupling in \cite[Remark 5.2]{BrenierMartingale}, where the marginals have support in $\mathbb R$. The symmetry operators $S$ and $\S$ permit to handle this case for $\mathbb R_+ ^*$-supported marginals.

To make this observation more precise, let $C(x,y)$ be a payoff satisfying $C_{xyy}<0$. Hence $\S^*(C)_{xyy}>0$ and by Proposition \ref{PropSymmetry2} we have
\begin{eqnarray*}
\overline{P}(\mu, \nu, C) &=& \overline{P}(S(\mu), S(\nu), \S^*(C)) \\
&=& \mathbb{E}^{\Q_L(S(\mu), S(\nu))}\left[\S^*(C)(X,Y)\right] \\
&=& \mathbb{E}^{\S\left(\Q_L(S(\mu), S(\nu))\right)}\left[C(X,Y)\right].
\end{eqnarray*}
Therefore, $\overline{P}(\mu, \nu, C)$ is attained by $\S\left(\Q_L(S(\mu), S(\nu))\right)$, which is equal to
$\Q_R(\mu,\nu)$ by Proposition \ref{LM-RM}. One can prove in a similar way that if $C_{xyy}>0$ (resp. $C_{xyy}<0$), the lower bound in \eqref{sup}  is attained by $\Q_R(\mu,\nu)$ (resp. $\Q_L(\mu,\nu)$).

\begin{remark}
\emph{We say that a payoff function $C$ is \textit{symmetric} if it satisfies $\S^*(C)=C$.\footnote{A way of constructing a symmetric payoff $C$ goes as follows: choose its values on $[0,1]\times\R_+^*$ first, then for $(x,y)\in(1,\infty)\times\R_+^*$, set $C(x,y)=yC(1/x,1/y)$.
One may easily check that $C$ satisfies $\S^*(C)=C$.}
For any symmetric payoff $C$ verifying the slightly relaxed generalized Spence-Mirrlees condition $C_{xyy}\geq 0$, we can use \eqref{sym-SM} to get $C_{xyy}(x,y) = -\frac{1}{x^2y^3} C_{xyy}(\frac{1}{x}, \frac{1}{y})$, hence $C_{xyy}=0$. Integrating with respect to $y$ twice and with respect to $x$ once, we see that $C$ is necessarily of the form $C(x,y) = \varphi(x) + \psi(y) + h(x)(y-x)$, for some functions $\varphi, \psi$ and $h$.}
\end{remark}

\subsection{Explicit constructions of left and right-monotone transference plans and change of numeraire}
\label{right}

In this section we briefly recall the explicit construction of a left-monotone transference plan performed by \cite{BrenierMartingale} and we show how the change of numeraire can be used to generate, essentially for free, the basic right-monotone transport plan from its left-monotone counterpart via the symmetry operator. We stress that Assumption \ref{assSingleMax} is still in force. The explicit characterization of $Q_L$ in \cite{BrenierMartingale} is described, for reader's convenience, in the following theorem.

\begin{theorem}
\label{ExistenceUniqL}
Let Assumption \ref{assSingleMax} hold. The left-monotone transference plan $Q_L$ is given by $Q_L (dx,dy)=\mu(dx)\mathcal L_L (x,dy)$ with transition kernel
$$
\mathcal{L}_L (x,\cdot) = \delta_x \1_{x\leq x_\star} + (q_L(x)\delta_{L_u(x)} + (1-q_L(x))\delta_{L_d(x)}) \1_{x > x^\star}, 
$$
where $q_L(x) := \frac{x-L_d(x)}{L_u(x)-L_d(x)}$, $x_\star\in\R_+^*$ is the unique maximizer of $\delta F$ and $ L_d,L_u$ are positive continuous functions on $(0, \infty)$, such that:
\begin{enumerate}[label=\roman*)]
\item $L_d(x)=L_u(x)=x$, for $x \leq x_\star$;
\item $L_d(x)<x<L_u(x)$,  for $x > x_\star$;
\item on the interval $(x^\star, \infty)$, $L_d$ is decreasing, $L_u$ is increasing.
\end{enumerate}
Moreover $L_d$ is the unique solution to
\begin{equation} 
\label{Td}
F_{\nu}^{-1}( F_{\mu}(x) + \delta F (L_d(x)) ) = G_{\nu}^{-1}( G_{\mu}(x) +\delta G(L_d(x)) ), \quad x>x_\star,
\end{equation}
and $L_u$ is given by the relation
\begin{equation}
\label{Lu-Ld} 
F_\nu(L_u(x)) = F_\mu(x) + \dF(L_d(x)), \quad x > x_\star .
\end{equation}
\end{theorem}
\begin{proof}
We refer to Theorem 4.5 in \cite{BrenierMartingale}. More details on the case of a single maximizer can be found in Section 3.4 therein.
\end{proof}

Now, using the fact that $\S(Q_L(S(\mu),S(\nu))) = Q_R(\mu,\nu)$ together with the characterization of the left-monotone transference plan given in the previous theorem, we can investigate how the quantities defining $Q_R$ and $Q_L$ are related to each other. Notice that, since both marginals have support in $\mathbb R_+ ^*$, the symmetry relation we use here is different than the one in Remark 5.2 in \cite{BrenierMartingale}.

\begin{proposition}
Let Assumption \ref{assSingleMax} hold. Then the right-monotone transference plan $Q_R$ is given by $Q_R (dx,dy) = \mu(dx)\mathcal L_R (x,dy)$ with transition kernel
$$
\mathcal{L}_R (x,\cdot) := \delta_x \1_{x\leq x^\star} + (q_R(x)\delta_{R_u(x)} + (1-q_R(x))\delta_{R_d(x)}) \1_{x>x^\star}
$$
where \begin{enumerate}
\item $x^\star =1/x^S_\star$ is the unique minimizer of $\delta F$;
\item $R_d(x) = \dfrac{1}{L^S_u(1/x)}$, $R_u(x) = \dfrac{1}{L^S_d(1/x)}$, for $x >0$;
\item the transition probability is given by 
$q_R(x) =\dfrac{x}{R_u(x)} ( 1-q^S_L(1/x))$, for $x>0$.
\end{enumerate}
\end{proposition}

\begin{proof} By Lemma \ref{symmetry}, if $\mu, \nu \in \mathcal{P}_1$ satisfy $\mu \preccurlyeq \nu$, then their images by the symmetry operator $S$ verify the same conditions, i.e : $S(\mu), S(\nu) \in \mathcal{P}_1$ and $S(\mu) \preccurlyeq S(\nu)$. By Remark \ref{ass-preserve} one has that $\delta F_S = \delta F_{S(\mu),S(\nu)}$ has a single local maximizer and Theorem \ref{ExistenceUniqL} gives that there exists a left-monotone transference plan $Q_L ^S := Q_L (S(\mu),S(\nu))$ characterized as in Theorem \ref{ExistenceUniqL}. 

To conclude, since we already know that $\S(Q_L ^S) = Q_R (\mu,\nu)$ (see Proposition \ref{LM-RM}), it suffices to check that the measure $\tilde Q$ defined as $\tilde Q (dx,dy) := \mu(dx) \mathcal L_R (dx,dy)$ with the kernel $\mathcal L_R$ defined as in the statement, satisfies
\[ \tilde Q [f(X,Y)] = \S(Q_L ^S)[f(X,Y)],\]
for all bounded measurable functions $f: (\mathbb R_+ ^*)^2 \to \mathbb R$. This can be done by direct computation using the formulas for $x^\star$, $R_d$ and $R_u$ given in the statement. The details are therefore omitted.
\end{proof}

\begin{remark} \emph{As a by-product of the previous proposition, we get the characterization of $Q_R$ in terms of a triplet $(x_\star, R_d, R_u)$, where $x_\star >0$ is the unique minimizer of $\delta F$ and $ R_d, R_u$ are positive continuous functions on $\mathbb R_+ ^*$, such that:
\begin{enumerate}[label=\roman*)]
\item $R_d(x)=R_u(x)=x$, for $x \geq x_\star$, and $R_d(x)<x<R_u(x)$,  for $x < x_\star$;
\item $R_d$ (resp. $R_u$) is increasing (resp. decreasing) on $(0, x_\star)$;
\item the transition kernel $\mathcal{L}_R$, i.e. $Q_R (dx,dy)=\mu(dx) \mathcal L_R(x,dy)$, is defined by 
$$
\mathcal{L}_R(x,\cdot) = \delta_x \1_{x\leq x_\star} + (q_R(x)\delta_{R_u(x)} + (1-q_R(x))\delta_{R_d(x)}) \1_{x>x_\star}
$$
where $q_R(x) := \frac{x-R_d(x)}{R_u(x)-R_d(x)}$.
\end{enumerate}
Finally, one can check that $R_d$ and $R_u$ are solutions to
\begin{eqnarray} \label{Ru}
F_\nu^{-1}\left( F_\mu(x) +\dF(R_u(x))\right) &=& G_\nu^{-1}\left( G_\mu(x) +\dG(R_u(x)) \right) \\
G_\nu(R_d (x) ) - G_\mu(x)  &=& G_\nu(R_u(x)) - G_\mu(R_u(x)). \end{eqnarray}}
\end{remark}

\section{The symmetric marginals case}
\label{sym}

In this section we look at the particular situation where the marginals $\mu,\nu$ satisfy $S(\mu)=\mu$ and $S(\nu)=\nu$. In this case we will say then that the marginals $\mu$ and $\nu$ are symmetric. 
Note that the use of the word `symmetry' in this context comes from the fact that the corresponding volatility smiles at each maturity are symmetric in log-forward moneyness. Symmetric models have been further studied by, e.g., \cite{CarrLee2009put} and \cite{mikeSym}. In particular, in \cite{CarrLee2009put} this concept is called put-call symmetry (PCS). They also give many examples of symmetric models, cf. \cite[Sections 3 and 4]{CarrLee2009put}.

The stochastic volatility models with zero correlation between the volatility and the spot are a classical example of a symmetric model. Consider a situation where $\mu$ and $\nu$ are the marginals at two consecutive times of some stock price process $S$ whose dynamics follows the stochastic volatility model
\begin{eqnarray*}
dS_t &=& S_t \sqrt{V_t} dW^1_t, \; S_0=1 \\
dV_t &=& \alpha(t,V_t) dt + \beta(t,V_t)dW^2_t
\end{eqnarray*}
where $W^1$ and $W^2$ are two independent Brownian motions. Then a simple application of Girsanov's theorem yields $S(\mu) = \mu$ and $S(\nu) = \nu$ (cf. \cite[Proposition 3.1]{rtouzi96}). This includes the Black-Scholes model as a special case.

An additional property satisfied by the symmetric models is given in the following proposition. Recall that $m$ (resp. $\tilde m$) denotes the unique maximizer (resp. minimizer) of $\delta F_{\mu,\nu}$.
\begin{proposition}
Assume that $\mu$ and $\nu$ are symmetric and let Assumption \ref{assSingleMax} hold. Then the unique minimizer $\tilde{m}$
satisfies $\tilde{m}>m$ and it is given by $\tilde{m} = \frac{1}{m}$. As a consequence $m<1$.
\end{proposition}

\begin{proof}
Let $m$ be the single maximizer of $\delta F_{\mu, \nu}$ and $\tilde{m}$ its minimizer, the existence of which is ensured by the convex order of $\mu$ and $\nu$.
We know from Remark \ref{ass-preserve} that the minimizer $\tilde{m}_S$ of $\delta F_{S(\mu), S(\nu)}$ verifies the relation $m = \frac{1}{\tilde{m}_S}$. Since $\mu$ and $\nu$ are symmetric, then $m = 1/\tilde{m}$.
Since $\mu \preccurlyeq \nu$, we have $m < \tilde{m}$, and consequently $m<1$.
\end{proof}

\begin{example}[The symmetric log-normal case]
\emph{We give an example of symmetric model, where the laws $\mu$ and $\nu$ are log-normal distributions
$$
\mu \sim \ln\mathcal{N}\left(-\frac{\sigma_\mu^2}{2}, \sigma_\mu^2\right), \;\nu \sim \ln\mathcal{N}\left(-\frac{\sigma_\nu^2}{2}, \sigma_\nu^2\right)
\quad \text{with}\; \sigma_\mu<\sigma_\nu.$$
Their probability densities and cumulative distribution functions are given by 
$$p_i(x) = \frac{1}{x\sqrt{2\pi}\sigma_i} \exp\left[ -\frac{(\ln(x)+\frac{1}{2}\sigma_i^2)^2}{2\sigma_i^2}\right],\quad
F_i(x) =  \frac{1}{2} \left[1 + \erf \left(\frac{\ln(x) + \frac{1}{2}\sigma_i^2}{\sqrt{2}\sigma_i}\right)\right] , \quad i=\mu, \nu,$$
where $\erf$ is the error function defined by $\erf(x) = \frac{2}{\sqrt{\pi}} \int_0^x e^{-\frac{t^2}{2}} dt$, $x\in\R$.
In this case, the maximum $m$ and minimum $\tilde{m}$ of $\delta F := F_\nu-F_\mu$ can be computed explicitly. Indeed, they are solutions in $y$ of the equation
$$ \ln (y)^2 = 2\dfrac{\sigma_\mu^2\sigma_\nu^2}{\sigma_\nu^2 - \sigma_\mu^2} \ln \left(\dfrac{\sigma_\nu}{\sigma_\mu}\right) + \dfrac{\sigma_\mu^2\sigma_\nu^2}{4} , $$ which gives
\begin{equation*}
m = \exp \left\{ -\left( 2\dfrac{\sigma_\mu^2\sigma_\nu^2}{\sigma_\nu^2 - \sigma_\mu^2} \ln\left(\dfrac{\sigma_\nu}{\sigma_\mu}\right) + \dfrac{\sigma_\mu^2\sigma_\nu^2}{4}\right)^{1/2}\right\} \quad \text{and}\; \tilde{m} = \dfrac{1}{m}.
\end{equation*}
Note that $m < 1 < \tilde{m}$.
The two figures \ref{FigLeft} and \ref{FigHK} below illustrate the left and right-monotone transference plans $(L_d,L_u)$ and $(R_d,R_u)$, and the basic three-points band decreasing transference plan by \cite{HobsonKlimmek2013}.
Figure \ref{FigLeft} gives the behaviour of the function $\delta F$, showing in particular the location of its maximum $m$ and minimum $\tilde{m}$.}

\begin{figure}[h!]
   \begin{center}
   \includegraphics[scale=0.8]{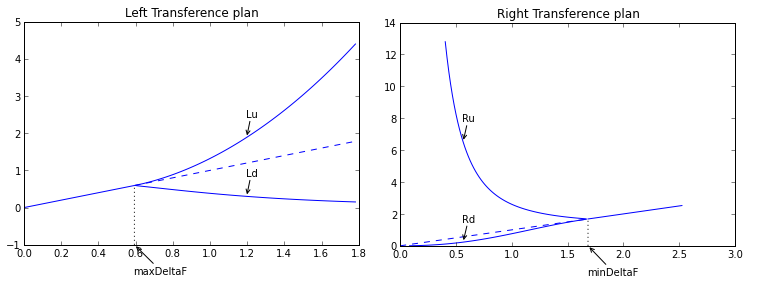} 
   \end{center}
   \caption{\label{FigLeft} Left and right monotone transference plan}
\end{figure}

\begin{figure}[h!]
   \begin{center}
   \includegraphics[scale=0.8]{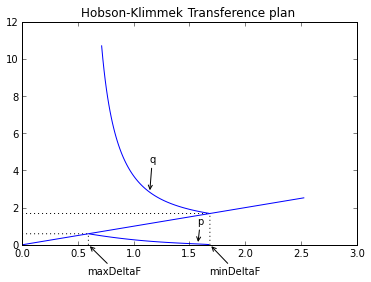} 
   \end{center}
   \caption{\label{FigHK} Hobson-Klimmek transference plan}
\end{figure}
\end{example}

\subsection{Symmetrized payoffs have a lower model risk}
In this subsection we show how the symmetry property of the marginals can be used to reduce the model risk of an option. The quantity $R(\mu, \nu, C) = \overline{P}(\mu, \nu, C) - \underline{P}(\mu, \nu, C)$ is a natural indicator of the model risk associated to a given payoff $C$. Obviously, model-risk free payoffs include payoffs $C$ which can be written as
$C(x,y) = \varphi(x) + \psi(y) +h(x) (y-x)$, since $R(C) = 0$ in this case. The following proposition shows that the converse is also true, under some conditions, even beyond the symmetric marginal case.

In the following proposition we will need some duality theory. We define the dual problems corresponding to $\underline P$ and $\overline P$ as
\[ \underline{D}(\mu, \nu, C) := \sup_{(\varphi, \psi, h)\in \underline{\mathcal{H}}} \mu(\varphi) + \nu(\psi), \quad \overline{D}(\mu, \nu, C) := \inf_{(\varphi, \psi, h)\in \overline{\mathcal{H}}} \mu(\varphi) + \nu(\psi),\]
where $\underline{\mathcal{H}}$ (resp. $\overline{\mathcal{H}}$) denotes the set of all triplets $(\varphi, \psi, h) \in \mathbb{L}^1(\mu) \times \mathbb{L}^1(\nu) \times \mathbb{L}^0$ such that $\varphi(x)+\psi(y)+h(x)(y-x)\leq C(x,y)$ (resp. $\varphi(x)+\psi(y)+h(x)(y-x)\geq C(x,y)$) for all $x,y\in \R_+^*$. Moreover we say that there is \emph{no duality gap} for the lower bound (resp. upper bound) if $\underline P (\mu,\nu,C) = \underline D(\mu,\nu,C)$ (resp. $\overline P(\mu,\nu,C) = \overline D(\mu,\nu,C)$). 

\begin{proposition}
Let $C$ be a payoff such that $R(C)=0$. Assume that the dual problem $\overline{D}$ is attained and that there is no duality gap. Then, there exist functions $\varphi\in \mathbb{L}^1(\mu), \;\psi\in \mathbb{L}^1(\nu),\; h\in \mathbb{L}^0$ such that
\begin{equation*}
C(x,y) = \varphi(x) + \psi(y) + h(x) (y-x), \;Q-a.e. \quad \forall Q \in \mathcal{M}(\mu,\nu).
\end{equation*}
\end{proposition}

\begin{proof}
Let $C$ be a payoff such that $R(C)=0$ and let there be no duality gap.
The property $R(C)=0$ implies that $\mathbb E^Q[C(X,Y)] = \overline{P}(\mu,\nu,C)$ for all $Q \in \mathcal{M}(\mu,\nu)$. Moreover since the dual problem is attained there exist dual functions $\varphi, \psi, h$ such that
$$\mu(\varphi) + \nu(\psi) = \overline{P}(\mu,\nu,C)$$ and
\begin{equation}
\label{ineq} 
C(x,y) \geq \varphi(x) + \psi(y) + h(x) (y-x), \quad Q - a.e.
\end{equation}
Since all $Q \in \mathcal{M}(\mu,\nu)$ have marginals $\mu$ and $\nu$ as well as the martingale property, we have 
$$\overline{P}(\mu,\nu,C) = \mathbb E^Q[C(X,Y)]= \mathbb E^Q[\varphi(X) + \psi(Y) + h(X) (Y-X)],\; \forall Q \in \mathcal{M}(\mu,\nu).$$
Consequently, we have
\begin{equation*}
\mathbb E^Q\left[C(X,Y) - \varphi(X) - \psi(Y) - h(X) (Y-X)\right] = 0, \; \forall Q \in \mathcal{M}(\mu,\nu).
\end{equation*}
which, combined with (\ref{ineq}), gives $C(x,y) = \varphi(x) + \psi(y) + h(x) (y-x)$, $Q$-a.e. for all $Q \in \mathcal{M}(\mu,\nu)$.
\end{proof}

\begin{remark}\label{duality} \emph{We recall that \cite{BeiglHLPenkner} consider the two-marginals martingale minimization problem $\underline P$ in (\ref{sup}) for upper semi-continuous payoffs $C$ with linear growth, and prove that there is no duality gap under some suitable conditions. Analogous results can be deduced for the primal maximisation problem. In general, the value functions of the corresponding dual problem is not always attained. The very recent paper \cite{beiglbock2015complete} proposes a quasi-sure relaxation of the dual problem, leading to an extension of ``no duality gap'' result to any Borel payoff with the existence of a dual optimizer.}
\end{remark}

Now, let the marginals $\mu$ and $\nu$ be symmetric and let $C$ be any continuous payoff with linear growth. By Proposition \ref{PropSymmetry2}, we have 
$$\overline{P}(\mu, \nu, C) = \overline{P}(\mu, \nu, \S^*(C)), \quad \underline{P}(\mu, \nu, C) = \underline{P}(\mu, \nu, \S^*(C)),$$
implying $ R(\mu, \nu, C) = R(\mu, \nu, \S^*(C))$. In particular, this gives $R(C_\alpha) \leq R(C)$ for payoffs $C_\alpha = \alpha C + (1-\alpha) \S^*(C)$ with $\alpha\in[0,1]$. In financial terms, this means that the new payoff $C_\alpha$ reduces the model risk. Note that $R(C_0) = R(C_1) = R(C)$. Moreover, we have $R(\S^*(C_\alpha)) = R(C_\alpha)$, and since $S$ is an involution, we get $R(C_{1-\alpha}) = R(C_\alpha)$. 

On the other hand, $C_{1/2} = (C +  \S^*(C))/2 = (C_\alpha + C_{1-\alpha})/2$, and because of the symmetry of $R(C_\alpha)$ around $1/2$ we get
\begin{eqnarray*}
R(C_{1/2}) &=& R\left( \dfrac{ C_\alpha + C_{1-\alpha}}{2} \right) \leq \frac{1}{2} R(C_\alpha) + \frac{1}{2} R(C_{1-\alpha})\\
& =& \frac{1}{2} R(C_\alpha) + \frac{1}{2} R(C_{\alpha}) = R(C_{\alpha}).
\end{eqnarray*}
Hence, $\alpha=1/2$ realizes the minimum model risk for the portfolio $C_\alpha$.

\section{Summary}\label{summary}
In this paper we introduce change of numeraire techniques in the two-marginals transport problems for positive martingales. In particular, we study the symmetry properties of \cite{HobsonKlimmek2013} optimal coupling under the change of numeraire, which exchanges type I with type II forward start straddle. As a consequence, we prove that the lower bound prices are attained for both options by the Hobson-Klimmek transference plan. On the other hand, relying on the construction of \cite{BrenierMartingale} of the optimal transference plan introduced by \cite{BeiglJuil}, we also show that the change of numeraire transformation exchanges the left and the right monotone transference plans, so that the latter can be viewed has a mirror coupling acting of the former under a change of numeraire for positive martingales with given marginals. We conclude this paper with some numerical illustrations in the symmetric log-normal marginals case.

\bibliographystyle{plainnat}

\bibliography{biblio}

\end{document}